\documentclass[a4paper]{article}
\usepackage{amssymb}
\usepackage{amsmath}
\usepackage{amsthm}
\usepackage{color}
\usepackage{amscd}
\usepackage{graphicx}
\usepackage{pdflscape}
\usepackage{longtable,booktabs}
\usepackage{float}
\usepackage[dvipsnames]{xcolor}
\usepackage{placeins}

\newtheorem{thm}{Theorem}[section]
\newtheorem*{mainthm}{Main Theorem}

\newtheorem{lemma}[thm]{Lemma}
\newtheorem{algorithm}[thm]{Algorithm}

\newtheorem{remark}[thm]{Remark}

\newtheorem{definition}[thm]{Definition}

\newcommand{\SL}{\mathrm{SL}}
\newcommand{\Sing}{\mathrm{Sing}}

\newcommand{\Supp}{\mathrm{Supp}}

\newcommand{\C}{\mathbb{C}}

\newcommand{\G}{\mathbb{G}}
\newcommand{\PP}{\mathbb{P}}
\newcommand{\wt}{\mathrm{wt}}

\newcommand{\lct}{\mathrm{lct}}

\newcommand{\Q}{\mathbb{Q}}
\newcommand{\Z}{\mathbb{Z}}
\newcommand{\R}{\mathbb{R}}
\newcommand{\A}{\mathbb{A}}
\newcommand{\Span}{\mathrm{Span}}

\begin{document}

\title{Uniformity of the minimal exponent on the GIT boundary of projective hypersurfaces}
\author{Yasutaka Shibata}
\date{}
\maketitle

\noindent E-mail address: \texttt{yasutaka@shibata-math.com}

\begin{abstract}
Let $X$ be a non-stable (in the sense of GIT) hypersurface of degree $d \ge 3$ in the projective space $\PP^{n}$. S.~G.~Park \cite{Par25} proved that the global minimal exponent of $X$ satisfies the inequality $\widetilde{\alpha}(X) \le \frac{n+1}{d}$. In this paper, we establish the uniform sharpness of this inequality: we prove that $\widetilde{\alpha}(X) = \frac{n+1}{d}$ holds for a general member $X$ of every irreducible component of the strictly semistable locus. Moreover, a general point of each irreducible component of the GIT boundary $\partial \mathcal{M}^{\mathrm{GIT}}$ admits a strictly semistable zero-weight representative with this minimal exponent. By the lower semicontinuity of the minimal exponent on the affine cone \cite{MP20a} and a theorem of Musta\c{t}\u{a}--Popa \cite{MP20b} for general linear combinations of monomials, the problem is reduced to computing the log canonical thresholds, at the cone vertex, of the monomial ideals associated with the maximal strictly semistable supports with respect to a maximal torus and their weight-zero parts. Since the barycenter of the degree-$d$ simplex lies in the convex hull of each of these supports, Howald's formula \cite{How01} forces these thresholds to be uniformly equal to $\frac{n+1}{d}$.
\end{abstract}

\section*{Introduction}

In \cite{Shi1,Shi2}, the author carried out a detailed analysis of the GIT boundary of the moduli spaces of cubic fivefolds and quintic threefolds. In the course of those investigations, a striking phenomenon was observed: the global minimal exponents of the general representatives considered for the different boundary components uniformly take the critical value $\frac{n+1}{d}$. The purpose of the present paper is to provide a conceptual explanation for this phenomenon and to prove that this uniformity holds in general for projective hypersurfaces of degree $d\geq3$ in $\PP^n$.

Let $\Sigma^{ss}$ denote the strictly semistable locus of degree-$d$ hypersurfaces in $\PP^n$ over $\C$, and let
$\pi:\Sigma^{ss}\to\partial\mathcal M^{\mathrm{GIT}}$
be the quotient map.

\begin{mainthm}
Assume $n\geq1$ and $d\geq3$.
A general member $X$ of every irreducible component of $\Sigma^{ss}$ satisfies
$\widetilde{\alpha}(X)=\frac{n+1}{d}$.
Moreover, for every irreducible component $B$ of
$\partial\mathcal M^{\mathrm{GIT}}$, there exist a nonempty open subset
$B^\circ\subseteq B$ and a nontrivial one-parameter subgroup
$\lambda:\G_m\to\SL(n+1)$ such that every $b\in B^\circ$ admits a
strictly semistable representative $[f]$ satisfying
\[
  \pi([f])=b,\qquad
  \lambda(t)\cdot f=f\quad(t\in\G_m),\qquad
  \widetilde{\alpha}(V(f))=\frac{n+1}{d}.
\]
\end{mainthm}

This is a consequence of the more precise statement in
Theorem~\ref{thm:main}.
The quotient-side conclusion is an existence statement for zero-weight
representatives. It does not assert that these representatives are polystable.

The key idea of our proof is to combine a finite covering by monomial families with a uniform computation of the minimal exponent at the vertex of the affine cone. We consider a general member $X=V(f)\subset\PP^n$ of one of the linear families constructed below and its affine cone $C(X)\subset\A^{n+1}$. Although the cone can have non-isolated singularities, its vertex $v$ is a point where the local minimal exponent on $C(X)$ attains its minimum. For such a general member, we show that this minimum is exactly $\frac{n+1}{d}$.

Indeed, for such a general member and any singular point $p\in\Sing(X)$, consider the ruling line $\ell\subset C(X)$ passing through $v$ and lying over $p$. Invariance under smooth pullback and the lower semicontinuity of the minimal exponent along $\ell$ \cite{MP20a} yield
\[
  \widetilde{\alpha}_{X,p}
  \geq \widetilde{\alpha}_v\bigl(C(X)\bigr)
  =\frac{n+1}{d}.
\]
Taking the minimum over all $p\in\Sing(X)$, we obtain the global lower bound
\[
  \widetilde{\alpha}(X)
  =\min_{p\in\Sing(X)}\widetilde{\alpha}_{X,p}
  \geq\frac{n+1}{d}.
\]

More generally, Park \cite[Theorem~6.1]{Par25} proves the formula
\[
  \widetilde{\alpha}_{v}(f)
  =
  \min\left\{\widetilde{\alpha}(X),\frac{n+1}{d}\right\}.
\]
The uniformity in our argument comes from the common barycenter condition
on the monomial supports in the finite covering.

The minimal exponent at the vertex $v$ is computed as follows. We first construct a finite covering using the maximal strictly semistable supports $I(\mathbf r_k)_{\geq0}$, and show that a general point on each boundary component can be represented by a defining polynomial $f$ with zero-weight support $\Supp(f)=I(\mathbf r_k)_{=0}$. Both supports contain the barycenter
$\eta=(\frac{d}{n+1},\ldots,\frac{d}{n+1})$
in their convex hulls. By a theorem of Musta\c{t}\u{a} and Popa \cite{MP20b}, for a general linear combination $f$ of the monomials in $I(\mathbf r_k)_{=0}$, the computation of the minimal exponent of $C(X)$ at $v$ is reduced to computing the log canonical threshold $\lct_v(\mathfrak a)$ of the associated monomial ideal $\mathfrak a$. Next, by Howald's formula \cite{How01}, $\lct_v(\mathfrak a)$ is given by the reciprocal $1/\tau$ of the parameter $\tau>0$ at which the diagonal ray $t(1,1,\ldots,1)$ intersects the boundary of the Newton polyhedron
\[
  P(\mathfrak a)
  =\mathrm{Conv}(I(\mathbf r_k)_{=0})+\R_{\geq0}^{n+1}.
\]
Since every monomial in $I(\mathbf r_k)_{=0}$ has degree $d$ and the barycenter $\eta$ lies in $\mathrm{Conv}(I(\mathbf r_k)_{=0})$, this intersection point is precisely $\eta$, which gives $\tau=\frac{d}{n+1}$ and hence $\lct_v(\mathfrak a)=\frac{n+1}{d}$. The same computation applies to the maximal strictly semistable support $I(\mathbf r_k)_{\geq0}$.

On the other hand, the supports of these families imply that $X$ is not GIT stable, so Park's theorem \cite{Par25} gives the upper bound
\[
  \widetilde{\alpha}(X)\leq\frac{n+1}{d}.
\]
Combining this with the lower bound above, we conclude that
$\widetilde{\alpha}(X)=\frac{n+1}{d}$.

\section{Numerical criterion for projective hypersurfaces}\label{num}
Throughout the paper, we work over $\C$ and fix integers $n \ge 1$ and $d \ge 3$. Hypersurfaces are regarded as the effective Cartier divisors defined by their equations; reducedness is not assumed. In this section, we review the numerical criterion for stability or semistability of projective hypersurfaces. We use the following notation.

\begin{itemize}
\item For a vector $\mathbf{x} \in \mathbb{Q}^{n+1}$, $\mathrm{wt}(\mathbf{x})=\sum_{k=0}^{n} x_{k}$ is called the weight of $\mathbf{x}$.

\item We define
\[
\begin{aligned}
\mathbb{Z}^{n+1}_{\ge 0}
&=\{\mathbf{x}\in\mathbb{Z}^{n+1}\mid x_k\ge0\ (k=0,1,\ldots,n)\},\\
\mathbb{Z}^{n+1}_{(d)}
&=\{\mathbf{x}\in\mathbb{Z}^{n+1}\mid\mathrm{wt}(\mathbf{x})=d\},\\
I&=\mathbb{Z}^{n+1}_{(d)}\cap\mathbb{Z}^{n+1}_{\ge0},
\end{aligned}
\]
where $I$ is the set of lattice points of the degree-$d$ simplex.

\item For $\mathbf{r} \in \mathbb{Q}^{n+1}$, we define $I(\mathbf{r})_{\ge 0}=\{ \mathbf{i} \in I \mid \mathbf{r} \cdot \mathbf{i} \ge 0 \}$, $I(\mathbf{r})_{> 0}=\{ \mathbf{i} \in I \mid \mathbf{r} \cdot \mathbf{i} > 0 \}$, and $I(\mathbf{r})_{=0}=\{ \mathbf{i} \in I \mid \mathbf{r} \cdot \mathbf{i} = 0 \}$, where $\cdot$ denotes the standard inner product of vectors.

\item We set $\eta=(\frac{d}{n+1},\frac{d}{n+1}, \cdots, \frac{d}{n+1}) \in \mathbb{Q}^{n+1}$ and call it the barycenter of the simplex $\mathrm{Conv}(I)$.

\item A vector $\mathbf{r} \in \mathbb{Z}^{n+1}$ is said to be reduced when there is no integer $\alpha$ such that $|\alpha| \ge 2$ and $\frac{1}{\alpha}\mathbf{r} \in \mathbb{Z}^{n+1}$.

\item Let $W = \mathbb{C}[x_{0},\cdots,x_{n}]_{d}$ be the vector space of homogeneous polynomials of degree $d$. For a polynomial $f=\sum_{\mathrm{wt}(\mathbf{i})=d}a_{\mathbf{i}}x^{\mathbf{i}} \in W$, we define the support of $f$ by $\mathrm{Supp}(f)=\{ \mathbf{i} \in I \mid a_{\mathbf{i}} \neq 0\}$.

\item The projective space $\PP(W)$ parametrizes hypersurfaces of degree $d$ in $\PP^{n}$.

\item We denote by $\PP(W)^{ss}$ and $\PP(W)^{s}$ the loci of GIT semistable and GIT stable hypersurfaces, respectively, with respect to the natural action of $\SL(n+1)$.

\item We set $\Sigma^{ss} = \PP(W)^{ss} \setminus \PP(W)^{s}$, $\mathcal{M}^{\mathrm{GIT}} = \PP(W)^{ss} /\!/ \SL(n+1)$, and $\partial \mathcal{M}^{\mathrm{GIT}} = \Sigma^{ss} /\!/ \SL(n+1)$, where $\Sigma^{ss}$ is the strictly semistable locus and $\partial \mathcal{M}^{\mathrm{GIT}}$ is the GIT boundary.

\item We denote the quotient map by $\pi \colon \PP(W)^{ss} \to \mathcal{M}^{\mathrm{GIT}}$, and use the same symbol for its restriction $\pi \colon \Sigma^{ss} \to \partial \mathcal{M}^{\mathrm{GIT}}$.

\end{itemize}

We fix a maximal torus $T$ of $\SL(n+1)$ and choose homogeneous
coordinates $x_0,\ldots,x_n$ that diagonalize its action on the space
of linear forms. We identify a one-parameter subgroup
(1-PS for short) $\lambda_{\mathbf r}:\mathbb G_m\to T$ with its
weight vector $\mathbf r=(r_0,\ldots,r_n)\in\mathbb Z^{n+1}_{(0)}$,
using the convention
\[
  \lambda_{\mathbf r}(t)\cdot x_j=t^{r_j}x_j.
\]
Thus
$\lambda_{\mathbf r}(t)\cdot x^{\mathbf i}
=t^{\mathbf r\cdot\mathbf i}x^{\mathbf i}$.

\begin{definition}
Let $s$ be a subset of $I$. We say that $s$ is not stable (resp. unstable) with respect to $T$ when $s \subseteq I(\mathbf{r})_{\ge 0}$ (resp. $s \subseteq I(\mathbf{r})_{> 0}$) for some nontrivial 1-PS $\mathbf{r} \neq 0$.
For $0 \neq f \in \mathbb{C}[x_{0},\cdots,x_{n}]_{d}$, we say that $f$ is not stable (resp. unstable) with respect to $T$ when $\mathrm{Supp}(f) \subseteq I$ is not stable (resp. unstable) with respect to $T$. 
\end{definition}

The following theorem is the numerical criterion for stability via the language of convex geometry.

\begin{thm}\label{numerical}
The projective hypersurface of degree $d$ defined by $0 \neq f \in W$ is not stable (resp. unstable) if and only if there exists an element $\sigma \in \SL(n+1)$ such that $f^{\sigma}$ is not stable (resp. unstable) with respect to $T$. 
In particular, $f$ is strictly semistable if and only if 
\begin{description}
\item[(1)] There exists $\sigma \in \SL(n+1)$ such that $f^{\sigma}$ is not stable with respect to $T$, and

\item[(2)] For any $\sigma \in \SL(n+1)$, $f^{\sigma}$ is semistable with respect to $T$.
\end{description}

\end{thm}
\begin{proof}
See Theorem 9.1 of \cite{Dol03}.
\end{proof}

\section{Maximal strictly semistable supports and the covering theorem}\label{sec:covering-theorem}

In this section, we characterize the maximal elements of the family of non-stable supports with respect to the maximal torus $T$, and use them to construct the finite covering of the strictly semistable locus in Theorem~\ref{thm:cover}.
To this end, we consider the family of subsets of $I$ ordered by inclusion, and describe all maximal elements of
\[
\mathcal{S}=\{ I(\mathbf{r})_{\ge 0} \mid \mathbf{r} \in \mathbb{Z}_{(0)}^{n+1} \setminus \{0\} \}.
\]
Before giving the characterization and an explicit algorithm to obtain them in finitely many steps (as carried out for concrete cases in \cite{Shi1,Shi2}), we remark that $I(\mathbf{r})_{\ge 0}$ and $I(\mathbf{r'})_{\ge 0}$ might be the same for two different vectors $\mathbf{r},\mathbf{r'} \in \mathbb{Z}^{n+1}_{(0)} \setminus \{0\}$. 

\begin{lemma}\label{lem:maximal}
Let $I(\mathbf{r})_{\ge 0}$ be a maximal element of $\mathcal{S}$, where $\mathbf{r} \in \mathbb{Z}^{n+1}_{(0)} \setminus \{0\}$. Then there exist $n-1$ elements $\mathbf{x}_{1},\mathbf{x}_{2},\cdots,\mathbf{x}_{n-1} \in I$ and a vector $\mathbf{r'} \in \mathbb{Z}^{n+1}_{(0)} \setminus \{0\}$ such that they satisfy the following three conditions:
\begin{description}
\item[(1)] The vector subspace $U$ of $\mathbb{Q}^{n+1}$ spanned by $\mathbf{x}_{1},\cdots,\mathbf{x}_{n-1},\eta$ over $\mathbb{Q}$ has dimension $n$.

\item[(2)] The vector $\mathbf{r'}$ is orthogonal to the subspace $U$ of $\mathbb{Q}^{n+1}$.

\item[(3)] $I(\mathbf{r})_{\ge 0}=I(\mathbf{r'})_{\ge 0}$.
\end{description}
\end{lemma}

\begin{proof}
Let us put $C=I(\mathbf{r})_{\ge 0} \cup \{\eta\}$. We consider its convex hull $\check{C}$ in the affine hyperplane
\[
H_d := \left\{\mathbf{x}\in\R^{n+1}\ \middle|\ \sum_{j=0}^{n}x_j=d\right\}.
\]
We first show that $\dim\check{C}=n$. Otherwise, the $\Q$-linear span of $C$ has dimension at most $n$, so there exists a nonzero integer vector $\mathbf{u}$ orthogonal to every element of $C$. Since $\mathbf{u}\cdot\eta=0$, we have $\mathbf{u}\in\Z^{n+1}_{(0)}$, and
\[
I(\mathbf{r})_{\ge0}\subseteq I(\mathbf{u})_{=0}\subseteq I(\mathbf{u})_{\ge0}.
\]
Choose $j$ such that $u_j>0$, which is possible because $\mathbf{u}\neq0$ and $\sum_j u_j=0$. If $\mathbf{e}_j$ is the $j$-th standard basis vector, then $d\mathbf{e}_j\in I(\mathbf{u})_{>0}$, whereas $d\mathbf{e}_j\notin I(\mathbf{r})_{\ge0}$. Thus $I(\mathbf{u})_{\ge0}$ strictly contains $I(\mathbf{r})_{\ge0}$, contradicting maximality.

Since $\check{C}\subseteq\{\mathbf{x}\mid\mathbf{r}\cdot\mathbf{x}\ge0\}$ and $\mathbf{r}\cdot\eta=0$, the point $\eta$ lies on the boundary of $\check{C}$ in $H_d$. Let $F$ be a facet of $\check{C}$ containing $\eta$. The facet is rational, and its supporting hyperplane in $H_d$ can be written as $\mathbf{r}'\cdot\mathbf{x}=0$ for some $\mathbf{r}'\in\Z^{n+1}_{(0)}\setminus\{0\}$, after clearing denominators. Choose its sign so that $\check{C}\subseteq\{\mathbf{x}\mid\mathbf{r}'\cdot\mathbf{x}\ge0\}$. Since $F$ has affine dimension $n-1$ and is spanned by points of $C\cap F$, we can choose $n-1$ points $\mathbf{x}_1,\ldots,\mathbf{x}_{n-1}\in I\cap F$ such that $\mathbf{x}_1,\ldots,\mathbf{x}_{n-1},\eta$ are affinely independent. They are also linearly independent, since all lie in $H_d$ and $d\neq0$. We have $I(\mathbf{r})_{\ge0}\subseteq I(\mathbf{r}')_{\ge0}$, so maximality gives $I(\mathbf{r})_{\ge0}=I(\mathbf{r}')_{\ge0}$.
\end{proof}

By this lemma, we can determine the set of maximal elements of $\mathcal{S}$ up to permutations of coordinates in finitely many steps using the following algorithm. 

\begin{algorithm}\label{alg:maximal}
Let $\mathcal{F}$ be the set of all ordered $(n-1)$-tuples
$\mathbf{x}=(\mathbf{x}_{1},\ldots,\mathbf{x}_{n-1})$ of \emph{pairwise distinct} points of $I$.
For $n=1$, $\mathcal F$ consists of the empty tuple.
Fix a total order on $\mathcal{F}$.
As initial data, set $\mathcal{S}'=\varnothing$ and let
$\mathbf{x}=(\mathbf{x}_{1},\ldots,\mathbf{x}_{n-1})$ be the minimum element of $\mathcal{F}$.
We modify $\mathcal{S}'$ by the following algorithm.
\begin{itemize}

\item Step 1.
Let $U\subset \mathbb{Q}^{n+1}$ be the $\mathbb{Q}$-subspace spanned by
$\mathbf{x}_{1},\ldots,\mathbf{x}_{n-1},\eta$.
If $\dim_{\mathbb{Q}}(U)=n$, then take a \emph{reduced} normal vector
$\mathbf{r}=(r_{0},\ldots,r_{n})\in \mathbb{Z}^{n+1}_{(0)}\setminus\{0\}$ of $U$
(i.e.\ $\mathbf{r}\cdot \mathbf{u}=0$ for all $\mathbf{u}\in U$)
and go to Step~2; otherwise, go to Step~5.

\item Step 2.
If $\mathbf{r}$ satisfies $r_{0} \ge \cdots \ge r_{n}$ or $r_{0} \le \cdots \le r_{n}$,
then go to Step~3; otherwise, go to Step~5.

\item Step 3.
If $r_{0} \ge \cdots \ge r_{n}$ (resp.\ $r_{0} \le \cdots \le r_{n}$),
add $I(\mathbf{r})_{\ge 0}$ (resp.\ $I(-\mathbf{r})_{\ge 0}$) to $\mathcal{S}'$
and go to Step~4.

\item Step 4.
Delete all elements of $\mathcal{S}'$ that are not maximal in $\mathcal{S}'$
(with respect to inclusion), and go to Step~5.

\item Step 5.
Replace the element $\mathbf{x}$ with the next element if $\mathbf{x}$ is not the maximum element,
and go to Step~1. Otherwise, stop the algorithm.

\end{itemize}
\end{algorithm}

For completeness, we justify the restriction to ordered weights and the maximality test. Lemma~\ref{lem:maximal} shows that every maximal support, after a permutation of coordinates, is produced by Steps~1--3. Moreover, if $r_0\ge\cdots\ge r_n$ and $I(\mathbf{r})_{\ge0}\subseteq I(\mathbf{s})_{\ge0}$, sorting the coordinates of $\mathbf{s}$ in decreasing order preserves this inclusion. Indeed, suppose $i<j$ and $s_i<s_j$, and let $\mathbf{s}'$ be obtained by exchanging $s_i$ and $s_j$. For $\mathbf{a}\in I(\mathbf{r})_{\ge0}$ with $a_i\ge a_j$, we have $\mathbf{s}'\cdot\mathbf{a}\ge\mathbf{s}\cdot\mathbf{a}\ge0$. If $a_i<a_j$, exchanging these two coordinates of $\mathbf{a}$ gives $\mathbf{a}'\in I(\mathbf{r})_{\ge0}$, and $\mathbf{s}'\cdot\mathbf{a}=\mathbf{s}\cdot\mathbf{a}'\ge0$. Repeating such exchanges proves the assertion. Strict inclusion is preserved as well, since coordinate permutations preserve the cardinality of a support. Thus every support produced by Steps~1--3 is contained in a maximal support produced by those steps, and Step~4 retains exactly the maximal supports, up to permutations of coordinates.

Every member of $\mathcal{S}$ is \emph{non-stable} with respect to $T$, but a maximal element $I(\mathbf{r})_{\ge 0} \in \mathcal{S}$ might still be \emph{unstable} with respect to $T$ (i.e., $I(\mathbf{r})_{\ge 0} \subseteq I(\mathbf{r}'')_{>0}$ for some other $\mathbf{r}'' \in \mathbb{Z}^{n+1}_{(0)} \setminus \{0\}$). The following lemma characterizes when $I(\mathbf{r})_{\ge 0}$ is strictly semistable with respect to $T$.

\begin{lemma}\label{lem:barycenter}
For a non-trivial 1-PS $\mathbf{r} \in \mathbb{Z}^{n+1}_{(0)} \setminus \{0\}$, the following three conditions are equivalent:
\begin{description}
\item[(1)] $I(\mathbf{r})_{\ge 0}$ is strictly semistable with respect to $T$ (that is, there is no $\mathbf{r}'' \in \mathbb{Z}^{n+1}_{(0)} \setminus \{0\}$ such that $I(\mathbf{r})_{\ge 0} \subseteq I(\mathbf{r}'')_{>0}$).
\item[(2)] $\eta \in \mathrm{Conv}(I(\mathbf{r})_{\ge 0})$.
\item[(3)] $\eta \in \mathrm{Conv}(I(\mathbf{r})_{=0})$.
\end{description}
\end{lemma}

\begin{proof}
To prove $\mathbf{(1)} \Leftrightarrow \mathbf{(2)}$, suppose first that $\eta\notin\mathrm{Conv}(I(\mathbf{r})_{\ge0})$. By strict separation and rational approximation, there exists $\mathbf{a}\in\Q^{n+1}$ such that $\mathbf{a}\cdot(\mathbf{i}-\eta)>0$ for every $\mathbf{i}\in I(\mathbf{r})_{\ge0}$. Subtracting $\frac{\sum_j a_j}{n+1}(1,\ldots,1)$ from $\mathbf{a}$ and clearing denominators gives a nonzero $\mathbf{r}''\in\Z^{n+1}_{(0)}$ such that $\mathbf{r}''\cdot\mathbf{i}>0$ for every such $\mathbf{i}$. Conversely, the existence of such a $\mathbf{r}''$ excludes $\eta$ from the convex hull, since $\mathbf{r}''\cdot\eta=0$.
The implication $\mathbf{(3)} \Rightarrow \mathbf{(2)}$ is obvious from $I(\mathbf{r})_{=0} \subseteq I(\mathbf{r})_{\ge 0}$.
To see $\mathbf{(2)} \Rightarrow \mathbf{(3)}$, suppose that $\eta = \sum_{\mathbf{i} \in I(\mathbf{r})_{\ge 0}} c_{\mathbf{i}} \mathbf{i}$ with $c_{\mathbf{i}} \ge 0$ and $\sum c_{\mathbf{i}} = 1$. Taking the inner product with $\mathbf{r}$, we have
\[
0 = \mathbf{r} \cdot \eta = \sum_{\mathbf{i} \in I(\mathbf{r})_{\ge 0}} c_{\mathbf{i}} (\mathbf{r} \cdot \mathbf{i}).
\]
Since $\mathbf{r} \cdot \mathbf{i} \ge 0$ for all $\mathbf{i} \in I(\mathbf{r})_{\ge 0}$ and $\mathbf{r} \cdot \mathbf{i} > 0$ for all $\mathbf{i} \in I(\mathbf{r})_{>0}$, we must have $c_{\mathbf{i}} = 0$ for all $\mathbf{i} \in I(\mathbf{r})_{>0}$. Thus $\eta \in \mathrm{Conv}(I(\mathbf{r})_{=0})$.
\end{proof}

By Lemma~\ref{lem:barycenter}, discarding from $\mathcal{S}'$ those elements $I(\mathbf{r})_{\ge 0}$ with
\[
\eta\notin\mathrm{Conv}(I(\mathbf{r})_{=0}),
\]
we obtain the complete list of maximal strictly semistable supports
\[
I(\mathbf{r}_{1})_{\ge 0},\ldots,I(\mathbf{r}_{M})_{\ge 0}
\]
with respect to $T$, up to permutations of coordinates.

For each $k=1,2,\dots,M$, we define the linear subspaces
\[
V_{k}:=\mathrm{Span}_{\C} \{ x^{\mathbf{i}} \mid \mathbf{i} \in I(\mathbf{r}_{k})_{\ge 0} \} \subset W, \qquad
W_{k}:=\mathrm{Span}_{\C} \{ x^{\mathbf{i}} \mid \mathbf{i} \in I(\mathbf{r}_{k})_{= 0} \} \subset V_{k},
\]
and set $X_{k} := \overline{\SL(n+1) \cdot \PP(V_{k})} \subset \PP(W)$. From the above arguments, we obtain the following covering theorem. 

\begin{thm}\label{thm:cover}
The strictly semistable locus $\Sigma^{ss} = \PP(W)^{ss} \setminus \PP(W)^{s}$ is covered by
\[
\Sigma^{ss}
= \bigcup_{k=1}^{M} \bigl(\SL(n+1) \cdot \PP(V_{k}) \cap \PP(W)^{ss}\bigr)
= \bigcup_{k=1}^{M} \bigl(X_{k} \cap \PP(W)^{ss}\bigr).
\]
\end{thm}

\begin{proof}
Let $[f] \in \Sigma^{ss}$. By Theorem~\ref{numerical}, after acting by an element of $\SL(n+1)$ and permuting the coordinates, there exists a reduced 1-PS $\mathbf{r} \in \Z^{n+1}_{(0)} \setminus \{0\}$ with $r_{0} \ge \dots \ge r_{n}$ such that $\Supp(f) \subseteq I(\mathbf{r})_{\ge 0}$. By Lemma~\ref{lem:maximal} and Algorithm~\ref{alg:maximal}, $I(\mathbf{r})_{\ge 0}$ is contained in a maximal element of $\mathcal{S}'$. Since $f$ is semistable, this maximal element must be semistable with respect to $T$, so it is one of $I(\mathbf{r}_{1})_{\ge 0}, \dots, I(\mathbf{r}_{M})_{\ge 0}$ by Lemma~\ref{lem:barycenter}. Thus $[f] \in \SL(n+1) \cdot \PP(V_{k})$ for some $k$. Conversely, any point in $\SL(n+1) \cdot \PP(V_{k}) \cap \PP(W)^{ss}$ is not stable by Theorem~\ref{numerical}, and since $\PP(W)^{s}$ is open, $X_{k} \cap \PP(W)^{ss} \subseteq \Sigma^{ss}$ also holds.
\end{proof}

\section{Reduction of the quotient image\\ to the weight-zero subspace}\label{sec:zero-weight}

Let $G := \SL(n+1)$ and take a nontrivial 1-PS
\[
  0 \neq \mathbf{r}=(r_0,\ldots,r_n) \in \mathbb{Z}^{n+1}_{(0)}.
\]
Write the corresponding weight decomposition of $W = \mathbb{C}[x_0,\ldots,x_n]_d$ as
\[
  W = \bigoplus_{m\in\mathbb{Z}} W_m(\mathbf{r}),
  \qquad
  W_m(\mathbf{r})
  := \Span_{\mathbb{C}}\{x^{\mathbf{i}} \mid \mathbf{i} \in I,\ \mathbf{r}\cdot \mathbf{i}=m\}.
\]
Set
\[
  V(\mathbf{r}) := \bigoplus_{m\geq 0} W_m(\mathbf{r})
  = \Span_{\mathbb{C}}\{x^{\mathbf{i}} \mid \mathbf{i} \in I(\mathbf{r})_{\ge 0}\}
\]
and
\[
  W(\mathbf{r}) := W_0(\mathbf{r})
  = \Span_{\mathbb{C}}\{x^{\mathbf{i}} \mid \mathbf{i} \in I(\mathbf{r})_{= 0}\}.
\]
If $H_{\mathbf{r}} := \lambda_{\mathbf{r}}(\mathbb{G}_m)$ denotes the image of the 1-PS $\lambda_{\mathbf{r}}$ associated with $\mathbf{r}$, then $W(\mathbf{r}) = W^{H_{\mathbf{r}}}$. Finally, define
\[
  X(\mathbf{r}) := \overline{G\cdot\PP(V(\mathbf{r}))} \subset \PP(W)
\]
and
\[
  \Phi(\mathbf{r}) := \pi\bigl(X(\mathbf{r})\cap \PP(W)^{ss}\bigr)
  \subset \partial\mathcal{M}^{\mathrm{GIT}}.
\]
Every point of $\PP(V(\mathbf{r}))$ is non-stable with respect to $T$. Hence $X(\mathbf{r})$ is contained in the non-stable locus, and
\[
  X(\mathbf{r})\cap \PP(W)^{ss} = X(\mathbf{r})\cap \Sigma^{ss}.
\]

\begin{lemma}[Zero-weight reduction in the good quotient]
\label{lem:zero-weight-reduction-in-quotient}
For every nonzero $\mathbf{r} \in \mathbb{Z}^{n+1}_{(0)}$, one has
\begin{equation}
\label{eq:zero-weight-quotient-image}
  \Phi(\mathbf{r})
  =
  \overline{
    \pi\bigl(\PP(V(\mathbf{r}))\cap \PP(W)^{ss}\bigr)
  }
  =
  \overline{
    \pi\bigl(\PP(W(\mathbf{r}))\cap \PP(W)^{ss}\bigr)
  },
\end{equation}
where the closures are taken in $\mathcal{M}^{\mathrm{GIT}}$.
\end{lemma}

\begin{proof}
Set
\[
  A(\mathbf{r}) := \PP(V(\mathbf{r}))\cap \PP(W)^{ss}.
\]
We first prove
\begin{equation}
\label{eq:prequotient-closure-to-quotient-closure}
  \Phi(\mathbf{r}) = \overline{\pi(A(\mathbf{r}))}.
\end{equation}
If $A(\mathbf{r})=\varnothing$, then $\PP(V(\mathbf{r}))$ is contained in the
unstable locus, which is closed and $G$-invariant. Therefore
$X(\mathbf{r})\cap \PP(W)^{ss}=\varnothing$, and
\eqref{eq:prequotient-closure-to-quotient-closure} is immediate.

Assume now that $A(\mathbf{r})\neq\varnothing$. Since $\PP(W)^{ss}$ is open in
$\PP(W)$, the subset $A(\mathbf{r})$ is a nonempty open, hence dense, subset of the
projective space $\PP(V(\mathbf{r}))$. Consequently,
\[
  X(\mathbf{r})\cap \PP(W)^{ss}
  =
  \overline{G\cdot A(\mathbf{r})}^{\,\PP(W)^{ss}}.
\]
Indeed, $G\times A(\mathbf{r})$ is dense in
$G\times\PP(V(\mathbf{r}))$, so the action morphism shows that
$G\cdot\PP(V(\mathbf{r}))$ is contained in the closure of
$G\cdot A(\mathbf{r})$.

The subset $X(\mathbf{r})\cap \PP(W)^{ss}$ is closed and $G$-invariant in
$\PP(W)^{ss}$. Hence its image under the good quotient is closed \cite{MFK94}. Since
$A(\mathbf{r})\subseteq X(\mathbf{r})\cap\PP(W)^{ss}$, we obtain
\[
  \overline{\pi(A(\mathbf{r}))}
  \subseteq
  \pi\bigl(X(\mathbf{r})\cap \PP(W)^{ss}\bigr).
\]
Conversely, continuity of $\pi$, together with its $G$-invariance,
gives
\[
\begin{aligned}
  \pi\bigl(X(\mathbf{r})\cap \PP(W)^{ss}\bigr)
  &\subseteq
  \overline{\pi(G\cdot A(\mathbf{r}))} \\
  &=
  \overline{\pi(A(\mathbf{r}))}.
\end{aligned}
\]
This proves \eqref{eq:prequotient-closure-to-quotient-closure}.

It remains to replace $V(\mathbf{r})$ by its weight-zero subspace $W(\mathbf{r})$. Let
$[f]\in A(\mathbf{r})$, and write
\[
  f = f_0 + f_{>0},
  \qquad
  f_0 \in W(\mathbf{r}),
  \qquad
  f_{>0} \in \bigoplus_{m>0} W_m(\mathbf{r}).
\]
With the convention fixed in Section~\ref{num}, we have
\[
  \lim_{t\to 0}\lambda_{\mathbf{r}}(t)\cdot f = f_0
\]
in the affine space $W$. Since $[f]$ is semistable, there exists a
homogeneous invariant
\[
  F \in \mathbb{C}[W]^G
\]
of positive degree such that $F(f)\neq 0$. The invariance of $F$ gives
\[
  F(f_0)
  =
  \lim_{t\to 0} F\bigl(\lambda_{\mathbf{r}}(t)\cdot f\bigr)
  = F(f)
  \neq 0.
\]
Thus $f_0\neq 0$ and $[f_0]\in \PP(W)^{ss}$.

The morphism
\[
  \mathbb{G}_m \longrightarrow \PP(W)^{ss},
  \qquad
  t \longmapsto \lambda_{\mathbf{r}}(t)\cdot[f],
\]
extends to $\mathbb{A}^1$ with value $[f_0]$ at $0$. Its composition
with $\pi$ is constant on $\mathbb{G}_m$. Since
$\mathcal{M}^{\mathrm{GIT}}$ is separated, the composition is constant on
$\mathbb{A}^1$, and therefore
\[
  \pi([f]) = \pi([f_0]).
\]
It follows that
\[
  \pi\bigl(\PP(V(\mathbf{r}))\cap \PP(W)^{ss}\bigr)
  \subseteq
  \pi\bigl(\PP(W(\mathbf{r}))\cap \PP(W)^{ss}\bigr).
\]
The reverse inclusion follows from $W(\mathbf{r})\subseteq V(\mathbf{r})$. Hence the two
images are equal. Combining this equality with
\eqref{eq:prequotient-closure-to-quotient-closure} proves
\eqref{eq:zero-weight-quotient-image}.
\end{proof}

For the maximal strictly semistable supports $I(\mathbf{r}_1)_{\ge 0}, \ldots, I(\mathbf{r}_M)_{\ge 0}$ obtained in Section~\ref{sec:covering-theorem}, we have
\[
  V(\mathbf{r}_k)=V_k,
  \qquad
  W(\mathbf{r}_k)=W_k,
  \qquad
  X(\mathbf{r}_k)=X_k,
\]
and we set $\Phi_k := \Phi(\mathbf{r}_k) = \pi(X_k \cap \PP(W)^{ss})$. Combining Theorem~\ref{thm:cover} and Lemma~\ref{lem:zero-weight-reduction-in-quotient}, we obtain the weight-zero version of the covering theorem for the GIT boundary.

\begin{thm}\label{thm:cover-zero-weight}
The GIT boundary $\partial\mathcal{M}^{\mathrm{GIT}} = \pi(\Sigma^{ss})$ is covered by
\[
  \partial\mathcal{M}^{\mathrm{GIT}}
  = \bigcup_{k=1}^{M} \Phi_k
  = \bigcup_{k=1}^{M} \overline{\pi\bigl(\PP(W_k)\cap \PP(W)^{ss}\bigr)}.
\]
In particular, every irreducible component of $\partial\mathcal{M}^{\mathrm{GIT}}$ is of the form $\Phi_k = \overline{\pi(\PP(W_k)\cap \PP(W)^{ss})}$ for some $k \in \{1, \ldots, M\}$.
\end{thm}

\begin{proof}
Applying the quotient map $\pi$ to the covering of $\Sigma^{ss}$ in Theorem~\ref{thm:cover} yields
\[
  \partial\mathcal{M}^{\mathrm{GIT}}
  = \pi(\Sigma^{ss})
  = \bigcup_{k=1}^{M} \pi\bigl(X_k \cap \PP(W)^{ss}\bigr)
  = \bigcup_{k=1}^{M} \Phi_k.
\]
By Lemma~\ref{lem:zero-weight-reduction-in-quotient}, each $\Phi_k$ coincides with $\overline{\pi(\PP(W_k)\cap \PP(W)^{ss})}$ and is a closed irreducible subset of $\partial\mathcal{M}^{\mathrm{GIT}}$ (whenever it is nonempty). Since $\partial\mathcal{M}^{\mathrm{GIT}}$ is a finite union of these closed irreducible subsets, every irreducible component of $\partial\mathcal{M}^{\mathrm{GIT}}$ must be equal to $\Phi_k$ for some $k$.
\end{proof}

\section{Uniformity of the global minimal exponent for general strictly semistable members}\label{sec:main}

Let $X = V(f) \subset \PP^{n}$ be a hypersurface of degree $d \ge 3$.
For $p\in X$, choose a local defining equation $g$ of $X$ in $\PP^n$ at $p$, and let $b_{g,p}(s)$ be its local Bernstein--Sato polynomial.
The local minimal exponent $\widetilde{\alpha}_{X,p}$ is the negative of the largest root of the reduced Bernstein--Sato polynomial $b_{g,p}(s)/(s+1)$, with the convention that $\widetilde{\alpha}_{X,p}=\infty$ when this polynomial is $1$.
In particular, $\widetilde{\alpha}_{X,p}=\infty$ at smooth points, and
\[
  \lct_p(\PP^n,X)=\min\{1,\widetilde{\alpha}_{X,p}\};
\]
see \cite{MP20a}.
The global minimal exponent of $X$ is defined by
\[
  \widetilde{\alpha}(X) := \min_{p \in X} \widetilde{\alpha}_{X,p} = \min_{p \in \Sing(X)} \widetilde{\alpha}_{X,p}
\]
with the convention that the minimum over the empty set is $\infty$ (see \cite{MP20a, Par25}). Recall the fundamental stability criterion due to S.~G.~Park.

\begin{samepage}
\begin{thm}[Park {\cite[Theorem~A]{Par25}}]\label{thm:park}
Let $X \subset \PP^{n}$ be a hypersurface of degree $d \ge 3$.
\begin{description}
\item[(1)] If $\widetilde{\alpha}(X) > \frac{n+1}{d}$, then $X$ is GIT stable.
\item[(2)] If $\widetilde{\alpha}(X) \ge \frac{n+1}{d}$, then $X$ is GIT semistable.
\end{description}
In particular, if $X$ is not GIT stable, then $\widetilde{\alpha}(X) \le \frac{n+1}{d}$.
\end{thm}
\end{samepage}

We now prove the main result of this paper, which establishes the uniform sharpness of Park's inequality across all maximal strictly semistable families.

\begin{thm}\label{thm:main}
Let $I(\mathbf{r}_{1})_{\ge 0}, \ldots, I(\mathbf{r}_{M})_{\ge 0}$ be the maximal strictly semistable supports with respect to $T$ obtained in Section~\ref{sec:covering-theorem}. For each $k \in \{1, \ldots, M\}$ and for a general polynomial $f \in W_{k}$ (resp.\ a general polynomial $f \in V_{k}$), the hypersurface $X = V(f) \subset \PP^{n}$ is strictly semistable and satisfies
\[
  \widetilde{\alpha}(X) = \frac{n+1}{d}.
\]
Consequently, $\widetilde{\alpha}(X) = \frac{n+1}{d}$ holds for a general member $X$ of every irreducible component of $\Sigma^{ss}$. Moreover, for every irreducible component $B$ of the GIT boundary $\partial\mathcal{M}^{\mathrm{GIT}}$, there exist an index $k$ and a nonempty open subset $B^{\circ}\subseteq B$ such that every $b\in B^{\circ}$ admits a strictly semistable representative $[f]\in\PP(W_k)$ with $\pi([f])=b$ and $\widetilde{\alpha}(V(f))=\frac{n+1}{d}$.
\end{thm}

\begin{proof}
Fix $k \in \{1, \ldots, M\}$ and let $S = I(\mathbf{r}_{k})_{=0}$ (resp.\ $S = I(\mathbf{r}_{k})_{\ge 0}$). Consider the monomial ideal
\[
  \mathfrak{a}_{S} := \bigl(x^{\mathbf{i}} \;\big|\; \mathbf{i} \in S\bigr) \subset \C[x_{0}, \ldots, x_{n}].
\]
For a polynomial $f = \sum_{\mathbf{i} \in S} a_{\mathbf{i}} x^{\mathbf{i}}$, let $C(X) = V(f) \subset \A^{n+1}$ be the affine cone over $X = V(f) \subset \PP^{n}$ with vertex $v = (0, \ldots, 0) \in \A^{n+1}$. Since $\mathfrak{a}_{S}$ is generated by monomials of degree $d \ge 3 > 1$, the order of $\mathfrak{a}_{S}$ at $v$ is $d > 1$. By a theorem of Musta\c{t}\u{a} and Popa \cite[Definition~5.3 and Example~5.13]{MP20b}, for a general linear combination $f$ of the monomial generators of $\mathfrak{a}_{S}$, the local minimal exponent $\widetilde{\alpha}_{v}(f)$ of the affine cone $C(X)$ at the vertex $v$ coincides with the log canonical threshold of $\mathfrak{a}_{S}$ at $v$:
\begin{equation}\label{eq:mp-monomial}
  \widetilde{\alpha}_{v}(f) = \lct_{v}(\mathfrak{a}_{S}).
\end{equation}
The cited result requires neither that $\mathfrak a_S$ be primary to the
maximal ideal at $v$ nor that the cone have an isolated singularity there.
The threshold in \eqref{eq:mp-monomial} is that of an ideal and can exceed $1$.

By Howald's formula \cite{How01} (see also \cite{CLS11}), the log canonical threshold of the monomial ideal $\mathfrak{a}_{S}$ is given by
\[
  \lct_{v}(\mathfrak{a}_{S}) = \frac{1}{\tau},
  \qquad \text{where} \quad
  \tau = \min\bigl\{\, t > 0 \;\big|\; t(1, 1, \ldots, 1) \in P(\mathfrak{a}_{S}) \,\bigr\},
\]
and $P(\mathfrak{a}_{S}) = \mathrm{Conv}(S) + \R^{n+1}_{\ge 0}$ is the Newton polyhedron of $\mathfrak{a}_{S}$. Since $\wt(\mathbf{i}) = d$ for all $\mathbf{i} \in S$, any point $\mathbf{y} \in P(\mathfrak{a}_{S})$ satisfies $\wt(\mathbf{y}) \ge d$. In particular, if $t(1, 1, \ldots, 1) \in P(\mathfrak{a}_{S})$, then
\[
  t(n+1) = \wt\bigl(t(1, 1, \ldots, 1)\bigr) \ge d,
  \qquad \text{so} \quad
  t \ge \frac{d}{n+1}.
\]
On the other hand, since $I(\mathbf{r}_{k})_{\ge 0}$ is strictly semistable with respect to $T$, Lemma~\ref{lem:barycenter} shows that the barycenter $\eta = \frac{d}{n+1}(1, 1, \ldots, 1)$ satisfies
\[
\eta\in\mathrm{Conv}(I(\mathbf{r}_{k})_{=0})
\subseteq\mathrm{Conv}(S)\subset P(\mathfrak{a}_{S}).
\]
Therefore, the diagonal ray $t(1, 1, \ldots, 1)$ enters $P(\mathfrak{a}_{S})$ precisely at the barycenter $\eta$, which gives $\tau = \frac{d}{n+1}$. Combining this with \eqref{eq:mp-monomial}, we obtain
\begin{equation}\label{eq:vertex-minexp}
  \widetilde{\alpha}_{v}(f) = \lct_{v}(\mathfrak{a}_{S}) = \frac{n+1}{d}.
\end{equation}

Next, we compare $\widetilde{\alpha}(X)$ with $\widetilde{\alpha}_{v}(f)$. For any singular point $p \in \Sing(X)$, choose a point $\tilde{p} \in C(X) \setminus \{v\}$ lying over $p$. The natural map $\A^{n+1}\setminus\{v\}\to\PP^n$ is a smooth $\G_m$-bundle, and its restriction $C(X)\setminus\{v\}\to X$ is also a $\G_m$-bundle. Thus the analytic germ $(C(X), \tilde{p})$ is isomorphic to the product $(X, p) \times (\A^{1}, 0)$, and $\widetilde{\alpha}_{X,p} = \widetilde{\alpha}_{C(X),\tilde{p}}$ by the invariance of the minimal exponent under smooth pullbacks \cite{MP20a}. Moreover, since $f$ is homogeneous, the local minimal exponent is constant along each scalar $\G_m$-orbit in $C(X)\setminus\{v\}$, and each such orbit contains $v$ in its closure. By the lower semicontinuity of the minimal exponent \cite[Theorem~E(2)]{MP20a}, we have
\[
  \widetilde{\alpha}_{X,p} = \widetilde{\alpha}_{C(X),\tilde{p}} \ge \widetilde{\alpha}_{v}(f) = \frac{n+1}{d}
\]
for all $p \in \Sing(X)$, which yields the lower bound
\begin{equation}\label{eq:lower-bound}
  \widetilde{\alpha}(X) \ge \frac{n+1}{d}.
\end{equation}

By Theorem~\ref{thm:park}(2), the inequality \eqref{eq:lower-bound} implies that $X$ is GIT semistable. On the other hand, since $\Supp(f) \subseteq S \subseteq I(\mathbf{r}_{k})_{\ge 0}$, Theorem~\ref{numerical} implies that $X$ is not GIT stable. Hence $X$ is strictly semistable, and Theorem~\ref{thm:park}(1) gives the reverse inequality $\widetilde{\alpha}(X) \le \frac{n+1}{d}$ (in particular, $\Sing(X) \neq \varnothing$). Consequently, we conclude that
\[
  \widetilde{\alpha}(X) = \frac{n+1}{d}.
\]
Finally, for each $k$, choose nonempty open subsets $U_k^V\subseteq\PP(V_k)$ and $U_k^W\subseteq\PP(W_k)$ on which the asserted strict semistability and equality hold. By Chevalley's theorem \cite[Th\'eor\`eme~1.8.4]{EGAIV1}, $G\cdot U_k^V$ is constructible. It is dense in $X_k$, since $G\times U_k^V$ is dense in $G\times\PP(V_k)$, and hence contains a nonempty open subset of $X_k$. The minimal exponent is invariant under changes of coordinates, so the equality holds throughout $G\cdot U_k^V$. Each nonempty $X_k\cap\PP(W)^{ss}$ is a closed irreducible subset of $\Sigma^{ss}$. The finite covering in Theorem~\ref{thm:cover} therefore gives the assertion for every irreducible component of $\Sigma^{ss}$.

Similarly, $\pi(U_k^W)$ is constructible and dense in $\Phi_k=\overline{\pi(\PP(W_k)\cap\PP(W)^{ss})}$, so it contains a nonempty open subset of $\Phi_k$. By Theorem~\ref{thm:cover-zero-weight}, every irreducible component $B$ of the GIT boundary equals some $\Phi_k$. Taking a nonempty open subset $B^{\circ}\subseteq\pi(U_k^W)$ proves the final assertion.
\end{proof}

\begin{remark}
The generality assumption in Theorem~\ref{thm:main} is essential.
The chordal cubic threefold $Y\subset\PP^4$ is strictly semistable
with closed orbit and satisfies
\[
  \widetilde{\alpha}(Y)=\frac32<\frac53,
\]
as recalled in \cite[Remark~7.6(2)]{Par25}.
\end{remark}

\end{document}